\documentclass[11pt, reqno, dvipdfmx]{amsart}
\usepackage[T1]{fontenc}
\usepackage{lmodern}
\usepackage{amssymb, amsthm, mathrsfs}
\usepackage{mathtools}
\usepackage{geometry}
\usepackage{enumitem}
\setlist[enumerate]{
  label=\textup{(\arabic*)},
  font=\normalfont,
  leftmargin=25pt
}
\setlist[itemize]{
  font=\normalfont,
  leftmargin=25pt
}

\usepackage[all]{xy}
\usepackage{xcolor}
\definecolor{darkblue}{rgb}{0.0, 0.0, 0.55}
\definecolor{darkmagenta}{rgb}{0.55, 0.0, 0.55}
\definecolor{darkcyan}{rgb}{0.0, 0.55, 0.55}

\usepackage[dvipdfmx]{hyperref}
\hypersetup{
  colorlinks = true,
  linkcolor  = darkblue,
  citecolor  = darkmagenta,
  urlcolor   = darkcyan
}

\newcommand{\al}{\alpha}
\newcommand{\be}{\beta}

\newcommand{\la}{\lambda}

\newcommand{\si}{\sigma}

\newcommand{\vph}{\varphi}

\newcommand{\PP}{{\mathbb P}}

\newcommand{\ZZ}{{\mathbb Z}}
\newcommand{\NN}{{\mathbb N}}

\def\rnum#1{\expandafter{\romannumeral #1}}
\def\Rnum#1{\uppercase\expandafter{\romannumeral #1}}

\renewcommand{\mod}{\operatorname{\mathsf{mod}}}

\newcommand{\grmod}{\operatorname{\mathsf{grmod}}}
\newcommand{\GrMod}{\operatorname{\mathsf{GrMod}}}
\newcommand{\CM}{\operatorname{\mathsf{CM}^{\mathbb Z}}}
\newcommand{\uCM}{\operatorname{\underline{\mathsf{CM}}^{\mathbb Z}}}
\newcommand{\qgr}{\operatorname{\mathsf{qgr}}}
\newcommand{\D}{\operatorname{\mathsf{D}}}
\newcommand{\Hom}{\operatorname{Hom}}

\newcommand{\Ext}{\operatorname{Ext}}

\newcommand{\mnull}{\operatorname{null}}
\renewcommand{\Im}{\operatorname{Im}}
\newcommand{\Ker}{\operatorname{Ker}}
\newcommand{\Coker}{\operatorname{Coker}}
\newcommand{\gldim}{\operatorname{gldim}}
\newcommand{\injdim}{\operatorname{injdim}}

\newcommand{\op}{\operatorname{op}}

\theoremstyle{plain} 
\newtheorem{thm}{Theorem}[section]
\newtheorem{cor}[thm]{Corollary}
\newtheorem{lem}[thm]{Lemma}
\newtheorem{prop}[thm]{Proposition}

\theoremstyle{definition}
\newtheorem{dfn}[thm]{Definition}
\newtheorem{ex}[thm]{Example}

\numberwithin{equation}{section}

\begin{document}

\title{A note on even Clifford algebras of skew quadric hypersurfaces}

\author{Tomoya Oshio}
\address{Department of Science and Technology, Graduate School of Medicine, Science and Technology, Shinshu University, 3-1-1 Asahi, Matsumoto, Nagano 390-8621, Japan}
\email{26hs601e@shinshu-u.ac.jp} 

\author{Kenta Ueyama}
\address{Department of Mathematics, Faculty of Science, Shinshu University, 3-1-1 Asahi, Matsumoto, Nagano 390-8621, Japan}
\email{ueyama@shinshu-u.ac.jp}


\subjclass[2020]{
16S37,
16S38,
16G50,
18G65.
}
\keywords{noncommutative quadric hypersurface,
even Clifford algebra,
skew polynomial algebra,
Cohen-Macaulay module}

\begin{abstract}
Let $S_\alpha = k\langle x_1,\dots,x_n\rangle /(x_i x_j - \alpha_{ij} x_j x_i)$ be a standard graded skew polynomial algebra over an algebraically closed field $k$ of characteristic not equal to $2$. We show the following results.

When $n$ is odd and $f = x_1x_2 + \cdots + x_{n-2}x_{n-1} + x_n^2$
is a normal element of $S_\alpha$, the even Clifford algebra of the skew quadric hypersurface $S_\alpha/(f)$ is isomorphic to a full matrix algebra $M_{2^{(n-1)/2}}(k)$, and the stable category $\underline{\mathsf{CM}}^{\mathbb Z}(S_\alpha/(f))$ of graded maximal Cohen-Macaulay modules over $S_\alpha/(f)$ is triangle equivalent to the derived category $\mathsf{D}^b(\mathsf{mod}\,k)$.

When $n$ is even and $f = x_1x_2 + \cdots + x_{n-1}x_n$
is a normal element of $S_\alpha$, the even Clifford algebra of $S_\alpha/(f)$ is isomorphic to $M_{2^{(n-2)/2}}(k)^2$, and the stable category $\underline{\mathsf{CM}}^{\mathbb Z}(S_\alpha/(f))$ of graded maximal Cohen-Macaulay modules over $S_\alpha/(f)$ is triangle equivalent  to the derived category $\mathsf{D}^b(\mathsf{mod}\,k^2)$.

As a consequence, $S_\alpha/(f)$ is of finite Cohen-Macaulay representation type in both cases. These results demonstrate that $S_\alpha/(f)$ is a natural noncommutative generalization of the homogeneous coordinate ring of a smooth quadric hypersurface.
\end{abstract}

\maketitle

\section{Introduction}
The representation theory of maximal Cohen-Macaulay modules is a very active area of research (see e.g.\,\cite{Au, CR, Iy, LW, Yo}).
In particular, the stable categories $\uCM(A)$ of $\ZZ$-graded maximal Cohen-Macaulay modules over (commutative and noncommutative) graded Gorenstein rings $A$ have been extensively investigated, especially from the viewpoint of tilting theory (see e.g.\,\cite{AIR, BIY, Han, IKU, IT, KST1, KST2, KMY, LZ, MUs, MY}).

In \cite{SV}, Smith and Van den Bergh studied the stable categories of graded maximal Cohen-Macaulay modules over noncommutative quadric hypersurfaces, by developing the method of Buchweitz, Eisenbud, and Herzog \cite{BEH}.
In particular, one of the important results is that, for a noncommutative quadric hypersurface $A$, they constructed a finite-dimensional algebra $C(A)$ and established a triangle equivalence between $\uCM(A)$ and the bounded derived category of finite-dimensional modules over $C(A)$. Because $C(A)$ can be regarded as an analogue of the even Clifford algebra associated with a quadratic form, it is referred to as the \emph{even Clifford algebra} of $A$.
Since then, noncommutative quadric hypersurfaces have been intensively studied in noncommutative algebraic geometry, often through their even Clifford algebras.
In particular, it is known that if $C(A)$ is semisimple, then the algebra $A$ has particularly nice properties (see \cite[Theorem 5.5]{MUk}).

Let $k$ be an algebraically closed field of characteristic not equal to $2$.
The prototypical example of an even Clifford algebra $C(A)$ arises when $A=S/(f)$, where $S=k[x_1,\dots,x_n]$ with $\deg x_i=1$ and $f=x_1^2+\cdots+x_n^2 \in S$.
In this case, $A$ is the homogeneous coordinate ring of a smooth quadric hypersurface in $\PP^{n-1}$, and $C(A)$ is given as follows (see e.g.\,\cite{Lee}):
\begin{align}\label{e.commCA}
C(A) \cong 
\begin{cases} 
M_{2^{(n-1)/2}}(k) & \text{if $n$ is odd}, \\ 
M_{2^{(n-2)/2}}(k)^2 & \text{if $n$ is even}.
\end{cases}
\end{align}
This fact is closely related to Kn\"orrer periodicity \cite{Kn}.
Indeed, from this observation and the result of Smith and Van den Bergh, one obtains a triangle equivalence
\begin{align}\label{e.commsCM}
\uCM(A) \simeq 
\begin{cases} 
\D^b(\mod k) \simeq \uCM(k[x]/(x^2))  &\quad \text{if $n$ is odd}, \\ 
\D^b(\mod k^2) \simeq \uCM(k[x,y]/(x^2+y^2)) &\quad \text{if $n$ is even}.
\end{cases}
\end{align}

In view of the development of the theory of even Clifford algebras and Kn\"orrer periodicity for noncommutative quadric hypersurfaces, the following results are known:
\begin{itemize}
\item \cite{MUk} Let $S$ be a noetherian AS-regular algebra and $f \in S_2$ a regular normal element. 
Suppose that there exists a graded algebra automorphism $\si$ of $S$ such that $\si(f)=f$ and $af=f\si^2(a)$ for all $a \in S$, and define a graded algebra automorphism $\hat\si$ of the Ore extension $S[u;\si]$ by $\hat\si|_S=\si$ and $\hat\si(u)=u$. 
Then it was shown that there exists a triangle equivalence $\uCM(S/(f)) \simeq \uCM(S[u;\si][v;\hat{\si}]/(f+u^2+v^2))$.

\item \cite{HY} Let $S$ be a noetherian Koszul AS-regular algebra and $f \in S_2$ a regular central element. 
It was observed that $C(S/(f))$ is closely related to the $\ZZ_2$-graded Clifford deformation of the Koszul dual $(S/(f))^!$.

\item \cite{HMY} Let $S$ and $T$ be noetherian Koszul AS-regular algebras with regular central elements $f \in S_2$ and $g \in T_2$. 
Assume that $S \otimes T$ is noetherian. 
Then the algebra $C(S \otimes T/(f \otimes 1 + 1 \otimes g))$ was investigated.

\item \cite{LSW} Let $S$ be a noetherian Koszul AS-regular algebra and $f \in S_2$ a regular central element.
Then $C(S_P[u,v;\si]/(f+u^2+v^2))$ was investigated, where $S_P[u,v;\si]$ is a graded double Ore extension of $S$.

\item \cite{HMM} Let $S$ be a $3$-dimensional noetherian Koszul Calabi-Yau algebra and $f \in S_2$ a regular normal element.
In this setting, the classification of $C(S/(f))$ was given.

\item \cite{HMW} Let $S$ be a $3$-dimensional noetherian Koszul AS-regular algebra and $f \in S_2$ a regular central element.
In this setting, the classification of $C(S/(f))$ was given.
\end{itemize}

In this paper, we study even Clifford algebras $C(S_{\al}/(f))$ of \emph{skew quadric hypersurfaces} $S_{\al}/(f)$,
where $S_{\al}= k\langle x_1, \dots, x_n \rangle /(x_ix_j -\al_{ij} x_jx_i)$ is a standard skew polynomial algebra and $f$ is a homogeneous normal element of degree $2$. (In this case, $f$ is automatically a regular element.) 
Before presenting our results, we recall a theorem of Higashitani and the second author \cite{HU} as prior work.

Consider the element $f=x_1^2+\dots+x_{n}^2$ of a standard graded skew polynomial algebra $S_{\al}= k\langle x_1, \dots, x_n \rangle /(x_ix_j -\al_{ij} x_jx_i)$.
One can check that $f$ is a normal element if and only if $f$ is a central element if and only if 
\begin{align}\label{e.pm1}
\al_{ij} \in \{1,-1\} \quad \text{for all $1\leq i,j \leq n$}. 
\end{align}
Assume that \eqref{e.pm1} holds. We define an $(n+1) \times (n+1)$ matrix $X_\al=(X_{ij})$ with entries in $\mathbb{F}_2=\{0, 1\}$ by
\begin{align}\label{e.pm1m}
X_{ij}=
\begin{cases}
0 & \quad \text{if $i=j$ or if $\al_{ij}=-1$ with $i\neq j$},\\ 1 & \quad \text{otherwise}. 
\end{cases}
\end{align}
The following theorem states that $C(S_\al/(f))$ and $\uCM(S_\al/(f))$ can be computed using the matrix $X_\al$.

\begin{thm}[{\cite[Theorem 1.3]{HU}, \cite[Lemma 3.6]{Ue}}] \label{t.HU}
Let $S_{\alpha}= k\langle x_1, \dots, x_n \rangle /(x_ix_j -\al_{ij} x_jx_i)$ be a standard graded skew polynomial algebra in $n$ variables and let 
$f=x_1^2+\dots+x_{n}^2 \in S_\al$.
Assume that $f$ is normal, equivalently, that \eqref{e.pm1} holds and define $X_\al=(X_{ij}) \in M_{n+1}(\mathbb{F}_2)$ as in \eqref{e.pm1m}.
Let $\ell=\mnull_{\mathbb{F}_2} X_\al$, $q=2^\ell$, and $s=2^{(n-\ell-1)/2}$. Then $C(S_\al/(f)) \cong M_{s}(k)^{q}$. Furthermore, we have a triangle equivalence
$\uCM(S_\al/(f)) \simeq \D^b(\mod k^q)$.
\end{thm}

Note that Theorem \ref{t.HU} was proved using combinatorics of graphs.

We now turn our attention to the following isomorphism in the commutative case:
\begin{align} \label{e.mot}
k[x_1,\dots,x_n]/(x_1^2+\cdots+x_n^2) \cong
\begin{cases}
k[x_1,\dots,x_n]/(x_1x_2+\cdots+x_{n-2}x_{n-1}+x_n^2) &\text{if $n$ is odd},\\
k[x_1,\dots,x_n]/(x_1x_2+\cdots+x_{n-1}x_{n}) &\text{if $n$ is even}.
\end{cases}
\end{align}
However, this isomorphism does not hold in general when $k[x_1,\dots,x_n]$ is replaced by a skew polynomial ring $S_{\al}=k\langle x_1, \dots, x_n \rangle /(x_ix_j -\al_{ij} x_jx_i)$.
Since Theorem \ref{t.HU} can be regarded as a result concerning a noncommutative analogue of the left-hand side of \eqref{e.mot}, in this paper we focus on a noncommutative analogue of the right-hand side of \eqref{e.mot}. 
The main theorem of this paper is as follows.

\begin{thm}\label{t.intromain}
Let $S_{\al}= k\langle x_1, \dots, x_n \rangle /(x_ix_j -\al_{ij} x_jx_i)$ be a standard graded skew polynomial algebra in $n$ variables and let 
\[
f=
\begin{cases}
x_1x_2+\cdots+x_{n-2}x_{n-1}+x_n^2 &\text{if $n$ is odd},\\
x_1x_2+\cdots+x_{n-1}x_{n} &\text{if $n$ is even}.
\end{cases}
\]
in $S_\al$. Assume that $f$ is normal. Then
\begin{align*}
C(S_\al/(f)) \cong 
\begin{cases} 
M_{2^{(n-1)/2}}(k) & \quad \text{if $n$ is odd}, \\ 
M_{2^{(n-2)/2}}(k)^2 & \quad \text{if $n$ is even}.
\end{cases}
\end{align*}
Furthermore, we have a triangle equivalence
\begin{align*}
\uCM(S_\al/(f)) \simeq 
\begin{cases} 
\D^b(\mod k)  &\quad \text{if $n$ is odd}, \\ 
\D^b(\mod k^2) &\quad \text{if $n$ is even}.
\end{cases}
\end{align*}
\end{thm}

Notice that we do not assume that $f$ is central.
Although Theorem \ref{t.HU} involves a triangulated category that does not appear in \eqref{e.commsCM}, Theorem \ref{t.intromain} yields the same conclusion as \eqref{e.commsCM}; in other words, it can be considered as a direct generalization of \eqref{e.commsCM}.

Since Theorem \ref{t.intromain} shows that 
$C(S_\al/(f))$ is semisimple, we obtain the following result.

\begin{cor}\label{c.intromain}
Let 
$S_\al$ and $f$ be as in Theorem \ref{t.intromain}, and assume that $f$ is normal. Then the following hold.

\begin{enumerate}
\item If $n$ is odd (resp.\,if $n$ is even), then $S_\al/(f)$ has exactly one (resp.\,two) non-projective graded maximal Cohen-Macaulay module(s), up to isomorphism and degree shift. In particular, $S_\al/(f)$ is of finite Cohen-Macaulay representation type.
\item The noncommutative projective scheme 
$\qgr S_\al/(f)$ associated to $S_\al/(f)$ in the sense of Artin-Zhang \cite{AZ} satisfies
$\gldim(\qgr S_\al/(f))=n-2$.
That is, $\qgr S_\al/(f)$ is smooth in the sense of Smith-Van den Bergh \cite{SV}. In particular, the derived category 
$\D^b(\qgr S_\al/(f))$ admits a Serre functor.
\end{enumerate}
\end{cor}

This paper is organized as follows.
In Section \ref{sec.pre}, we fix notation and review background material needed for the proof of Theorem \ref{t.intromain}.
Sections \ref{sec.odd} and \ref{sec.even} contain the proof of Theorem \ref{t.intromain} in the cases of odd and even numbers of variables, respectively.
In Section \ref{sec.ce}, we prove Corollary \ref{c.intromain} and give an example.

\section{Preliminaries} \label{sec.pre}

Throughout this paper, let $k$ be an algebraically closed field with $\operatorname{char} k \neq 2$. All vector spaces and algebras are over $k$.
For a ring $R$, let $R^{\op}$ denote the opposite algebra of $R$.
We denote by $\mod R$  the category of finitely generated right $R$-modules, and we identify $\mod R^{\op}$ with the category of finitely generated left $R$-modules. 

\subsection{Graded maximal Cohen-Macaulay modules}
Let $A$ be a connected graded algebra, that is, $A=\bigoplus_{i \in \NN} A_i$ with $A_0 =k$.
We write $\GrMod A$ for the category of graded right $A$-modules $M=\bigoplus_{i\in\ZZ}M_i$ and degree-preserving $A$-module homomorphisms. For $M \in \GrMod A$ and $j \in \ZZ$,
we define the \emph{shift} $M(j) \in \GrMod A$ by $M(j)_i = M_{j+i}$.
For $M, N\in \GrMod A$, we define $\Ext^i_A(M, N):=\bigoplus _{j \in \ZZ}\Ext^i_{\GrMod A}(M, N(j))$.

The following classes of algebras play a central role in noncommutative algebraic geometry. 

\begin{dfn} 
Let $A$ be a noetherian connected graded algebra. We say that $A$ is \emph{AS-regular} (resp.\,\emph{AS-Gorenstein}) of dimension $n$ if
\begin{enumerate}
\item $\gldim A = n < \infty$ (resp.\,$\injdim_A A = \injdim_{A^{\op}} A = n < \infty$), and
\item $\Ext^i_A(k ,A) \cong \Ext^i_{A^{\op}}(k ,A) \cong
\begin{cases}
0 & \text{if } i \neq n,\\
k(\ell) \text{ for some } \ell \in \ZZ & \text{if } i = n.
\end{cases}$
\end{enumerate}
\end{dfn}

\begin{ex}
A \emph{standard graded skew polynomial algebra} in $n$ variables is a graded algebra
\[ S_\al = k\langle x_1, \dots, x_n \rangle /(x_ix_j -\al_{ij} x_jx_i \mid 1\leq i,j \leq n) \]
with $\deg x_i=1$ for all $1\leq i \leq n$, where $\alpha =(\alpha_{ij}) \in M_n(k)$ is a matrix satisfying $\alpha_{ii}=\alpha_{ij}\alpha_{ji}=1$ for all $1\leq i,j \leq n$.
It is well-known that such an algebra is a noetherian Koszul AS-regular domain of dimension $n$.
\end{ex}

Let $A$ be a noetherian AS-Gorenstein algebra.
We denote by $\grmod A$ the full subcategory of $\GrMod A$ consisting of finitely generated graded modules.
A graded module $M \in \grmod A$ is called \emph{maximal Cohen-Macaulay} if $\Ext^i_A(M, A)=0$ for all $i>0$.
Let $\CM(A)$ denote the full subcategory of $\grmod A$ consisting of graded maximal Cohen-Macaulay modules. Then $\CM(A)$ is a Frobenius category.
The \emph{stable category} of graded maximal Cohen-Macaulay modules, denoted by $\uCM(A)$, has the same objects as $\CM(A)$,
and the morphism space is given by
\[ \Hom_{\uCM(A)}(M, N) = \Hom_{\CM(A)}(M,N)/P(M,N), \]
where $P(M,N)$ is the subspace of degree-preserving $A$-module homomorphisms that factor through a graded projective module. Since $\CM(A)$ is a Frobenius category, $\uCM(A)$ admits the canonical structure of a triangulated category (see \cite{Hap}).

\subsection{Even Clifford algebras of noncommutative quadric hypersurfaces}

In this subsection, we recall the even Clifford algebras of noncommutative quadric hypersurfaces. Although these algebras were originally introduced by Smith and Van den Bergh \cite{SV}, for the purpose of this work we present a slightly generalized version due to Mori and the second author \cite{MUk}.

We first fix some basic notation. Let $A=T(V)/(R)$ be a quadratic algebra, where $T(V)=\bigoplus_{i\in\NN} V^{\otimes i}$ is the tensor algebra on a finite-dimensional vector space $V$, and $R$ is a subspace of $T(V)_2=V\otimes_k V$.
Then the \emph{quadratic dual} $A^!$ of $A$ is defined as $T(V^*)/(R^\perp)$, where $V^*$ is the $k$-linear dual of $V$, and $R^\perp=\{\mu \in  T(V^*)_2=V^*\otimes_k V^* \mid \mu(r)=0\ \text{for all}\ r \in R \}$. Note that we identify $(V\otimes_k V)^* \cong V^*\otimes_k V^*$ via $(\psi_1\otimes \psi_2)(v_1\otimes v_2)=\psi_1(v_1)\psi_2(v_2)$ for $\psi_1,\psi_2 \in V^*, v_1,v_2\in V$.

A connected graded algebra $A$ is called \emph{Koszul} if $k \in \GrMod A$ has a linear free resolution.
Suppose $A$ is Koszul. Then it is well-known that $A$ is quadratic, $A^!$ is also Koszul, and $A^!$ is isomorphic to the Yoneda algebra $(\bigoplus_{i\in\NN}\Ext_A^i(k,k))^{\op}$.
In this case, $A^!$ is also called the \emph{Koszul dual} of $A$.

\begin{dfn}\label{dfn.qh}
A connected graded algebra $A$ is called a \emph{noncommutative quadric hypersurface (ring)} of dimension $n-1$ if $A$ is of the form $A=S/(f)$, where $S$ is a noetherian Koszul AS-regular algebra of dimension $n \geq 1$ and $f \in S$ is a homogeneous regular normal element of degree $2$.
\end{dfn}

Let $S_{\al}=k\langle x_1, \dots, x_n \rangle /(x_ix_j -\al_{ij} x_jx_i)$ be a standard skew polynomial algebra, and let $f\in S_{\al}$ be a homogeneous normal element of degree $2$. Since $S_\al$ is a domain, $f$ is a regular element. We call $A=S_{\al}/(f)$ a \emph{skew quadric hypersurface}.
This algebra is the main object of study in this paper.

\begin{prop}\label{p.w}
Let $A=S/(f)$ be a noncommutative quadric hypersurface of dimension $n-1$.
\begin{enumerate}
\item \textnormal{(\cite[Lemma 2.4]{MUk})} $A$ is a noetherian Koszul AS-Gorenstein algebra of dimension $n-1$.
\item \textnormal{(\cite[Corollary 1.4]{ST})} There exists a homogeneous regular normal element $w \in A^!$ of degree $2$ such that $A^!/(w)=S^!$.
\end{enumerate}
\end{prop}

Let $A=S/(f)$ be a noncommutative quadric hypersurface.
Take a regular normal element $w \in A^!_2$ as in Proposition \ref{p.w}(2).
Then there exists a graded algebra automorphism $\nu_w$ of $A^!$ such that $aw = w\nu_w (a)$ for all $a \in A^!$. We call $ \nu_w$ the \emph{normalizing automorphism} of $w$. Then we obtain the $\ZZ$-graded localization $A^![w^{-1}]$, whose elements are written in the form $aw^{-i}$ with $a \in A^!$ and $i \in \NN$, and whose algebra structure is given by
\begin{itemize}
	\item (addition) $aw^{-i}+a'w^{-j}=(aw^j+a'w^i)w^{-i-j}$
	\item (multiplication) $(aw^{-i})(a'w^{-j})=a\nu_w^i(a')w^{-i-j}$
	\item (grading) $\deg(aw^{-i})=\deg a-2i$
\end{itemize}
for all $a,a'\in A^!$ and $i,j \in \NN$.

\begin{dfn}
With the above notation, for a noncommutative quadric hypersurface $A=S/(f)$, the \emph{even Clifford algebra} of $A$ is defined as
\[ C(A) := A^![w^{-1}]_0.\]
\end{dfn}

\begin{prop}[{\cite[Lemma 4.13(1)]{MUk}}]\label{p.dim}
Let $A=S/(f)$ be a noncommutative quadric hypersurface of dimension $n-1$. Suppose that $S$ has Hilbert series $(1-t)^{-n}$.
Then $\dim_k C(A)=2^{n-1}$.
\end{prop}

The following theorem shows the importance of the even Clifford algebras.

\begin{thm}[{\cite[Lemma 4.13(4)]{MUk}; see also \cite[Proposition 5.2]{SV}}]\label{t.SV}
Let $A=S/(f)$ be a noncommutative quadric hypersurface.
Then we have a triangle equivalence
\[
\uCM(A) \simeq \D^b(\mod C(A)^{\op}),
\]
where $\D^b(\mod C(A)^{\op})$ is the bounded derived category of $\mod C(A)^{\op}$.
\end{thm}

\section{Proof of Theorem \ref{t.intromain}: the case of an odd number of variables}\label{sec.odd}

This section is devoted to proving Theorem \ref{t.intromain} in the case where the number of variables is odd. Throughout this section,

\begin{itemize}
\item $n:=2m+1$ with $m \geq 0$,
\item $S_\al = k\langle x_1, \dots, x_{n} \rangle /(x_ix_j -\al_{ij} x_jx_i)$ is a standard graded skew polynomial algebra in $n=2m+1$ variables, 
\item $f:=x_1x_2+\cdots+x_{2m-1}x_{2m}+x_{n}^2 \in S_\al$,
\item $A_\al :=S_\al /(f)$,
\item $A_{\textnormal{comm}}:=k[x_1,\dots,x_n]/(x_1x_2+\cdots+x_{2m-1}x_{2m}+x_n^2)$,
that is, this is the special case of $S_\al /(f)$ where $\al_{ij}=1$ for all $1\leq i,j\leq n$.
\end{itemize}

\begin{lem}\label{l.normalo}
The element $f \in S_\al$ is normal if and only if 
\begin{align}\label{e.normalo}
\al_{2t-1,2s}\al_{2t,2s}=\al_{2s-1,2t-1}\al_{2s-1,2t}=\al_{2s-1,2s}=\al_{2s-1,n}^2 \quad \text{and} \quad
\al_{n,2s-1}\al_{n,2s}=1
\end{align}
for all $1\leq s,t\leq m$ with $s\neq t$.
\end{lem}

\begin{proof}
$(\Rightarrow)$ Suppose that $f$ is normal. Then, for each $1\leq i\leq n$, there exist $\la_{i1},\dots,\la_{in}\in k$ such that $x_if=f(\sum^{n}_{j=1}\la_{ij}x_j)$.
Since $\{x_1^{a_1}x_2^{a_2}\cdots x_{n}^{a_n} \mid a_{1}, a_{2}, \dots, a_n \geq 0\}$ forms a $k$-basis of $S_\al$, it follows that $\la_{ij}=0$ for $j\neq i$, so
we have $x_if=\la_{i}fx_i$, where we set $\la_{i}:=\la_{ii}$. 
Moreover, it is easily checked that
\begin{align}\label{e.normalcalo}
\begin{split}
&\la_{2s-1}=\al_{2s-1,2t-1}\al_{2s-1,2t}=\al_{2s-1,2s}=\al_{2s-1,n}^2,\quad \la_{2s}=\alpha_{2s,2t-1}\alpha_{2s,2t}= \alpha_{2s,2s-1}= \alpha_{2s,n}^2,\\
&\la_{n}= \al_{n,2s-1}\al_{n,2s}= 1
\end{split}
\end{align}
for all $1\leq s,t\leq m$ with $s\neq t$. Thus we have \eqref{e.normalo}.

$(\Leftarrow)$ Suppose that $\eqref{e.normalo}$. Define $\la_{2s-1}, \la_{2s}$ (for $1\leq s\leq m$) and $\la_{n}$ as in \eqref{e.normalcalo}.
Then one can check that $x_if=\la_{i}fx_i$ for all $1 \leq i\leq n$, so $f$ is a normal element.
\end{proof}

\begin{lem} \label{l.o1}
Assume that $f \in S_{\al}$ is normal. Let $L=\{(2s-1,2s) \mid 1\leq s \leq m \}$.
\begin{enumerate}
\item $S_\al^!$ is isomorphic to $k\langle x_1, \dots, x_{n} \rangle$ with relations
\begin{align*}
x_ix_j+\al_{ji}x_jx_i \;\; (1\leq i<j \leq n), \quad x_i^2 \;\; (1\leq i \leq n).
\end{align*}
\item $A_\al^!$ is isomorphic to $k\langle x_1, \dots, x_{n} \rangle$ with relations
\begin{align*}
&x_ix_j+\al_{ji}x_jx_i \;\;(1\leq i<j \leq n,\ (i,j) \not\in L),\\ 
&x_{i}x_{j}+\al_{ji}x_{j}x_{i}-x_{n}^2 \;\; ((i,j) \in L), \quad 
x_i^2 \;\;(1\leq i \leq 2m).
\end{align*}
\item $w:=x_{n}^2\in A_\al ^!$ is a normal element such that $A_\alpha^!/(w) \cong S_\alpha^!$.
\item $C(A_\al)=A_\al^![w^{-1}]_0$ is isomorphic to $k\langle z_1, \dots, z_{2m} \rangle$ with relations
\begin{align}\label{e.C(A)o}
\begin{split}
&z_iz_j+\widehat\al_{ji}z_jz_i \;\;(1\leq i<j \leq 2m,\ (i,j) \not\in L), \\
&z_iz_j+z_jz_i-1 \;\; ((i,j) \in L), \quad  z_i^2 \;\;(1\leq i \leq 2m),
\end{split}
\end{align}
where $\widehat\al_{ji}:=\al_{nj}\al_{ji}\al_{in}$.
\end{enumerate}
\end{lem}

\begin{proof}
(1) and (2) follow from a direct computation.

(3) Since $x_iw=\al_{ni}^2wx_i$ for $1\leq i\leq 2m$ and $x_{n}w=wx_{n}$, we see that $w$ is normal.
The last isomorphism immediately follows from (1) and (2).

(4) We define an algebra homomorphism $\vph: k\langle z_1, \dots, z_{2m} \rangle \to C(A_\al)$ by 
\[
z_{2s-1} \mapsto x_{2s-1}x_{n}w^{-1},\;\;
z_{2s}\mapsto -\al_{2s, n}x_{2s}x_{n}w^{-1} \quad\text{for $1\leq s \leq m$.}
\]
Since $w^{-1}x_{n}=x_{n}w^{-1}$, we have
\begin{align*}
x_ix_{n}w^{-1} x_jx_{n}w^{-1} 
&= -\al_{nj}x_ix_{n}w^{-1} x_{n}x_jw^{-1}
= -\al_{nj}x_ix_{n}^2w^{-1}x_jw^{-1} 
=-\al_{nj}x_ix_jw^{-1} 
\end{align*}
for any $1 \leq i, j \leq 2m$.
Thus we see that $\{x_ix_nw^{-1} \mid 1\leq i \leq2m \}$ generates $C(A_\al)$, and hence $\vph$ is surjective.
Moreover, if $(i,j)=(2s-1,2s) \in L$, then
$\al_{nj}\al_{ji}\al_{in}=\al_{n, 2s-1}^2\al_{2s-1, n}^2=1$
by \eqref{e.normalo}. Using the above arguments, one can verify that \eqref{e.C(A)o} lie in $\Ker \vph$. This yields that $\vph$ induces a surjective algebra homomorphism from $k\langle z_1, \dots, z_{2m} \rangle $ subject to the relations \eqref{e.C(A)o} to $C(A_\al)$.
Using Proposition \ref{p.dim}, we see that both algebras have dimension $2^{2m}$, so the induced homomorphism is an isomorphism.
 \end{proof}

For the remainder of this section, assume that 
$f \in S_\al$ is normal, and use the following notation.
\begin{itemize}
\item $L:=\{(2s-1,2s) \mid 1\leq s \leq m \}$.
\item For $1\leq i,j \leq 2m$, define $\widehat\al_{ij}:=\al_{ni}\al_{ij}\al_{jn}$.
\item For $1\leq r<s \leq m$, the elements $F_{sr}^{(1)},F_{sr}^{(2)} \in C(A_\al)$ are defined by
\begin{align*}
&F^{(1)}_{sr} =
1-(1-\widehat\alpha_{2s,2r})z_{2r-1}z_{2r},\\
&F^{(2)}_{sr} =
1-(1-\widehat\alpha_{2r,2s})z_{2r-1}z_{2r},
\end{align*}
where we identify the isomorphism obtained in Lemma \ref{l.o1}(4).
\end{itemize}

\begin{lem}\label{l.o2}
Assume that $f \in S_\al$ is normal.
\begin{enumerate}
\item $\widehat\alpha_{2r-1,i}\widehat\alpha_{2r,i}=\widehat\alpha_{i,2r-1}\widehat\alpha_{i,2r}= 1$ for $1\leq r \leq m$ and $1\leq i \leq 2m$.
\item In $C(A_\al)$, 
we have 
\[
F^{(1)}_{sr}z_i=
\begin{cases}
\widehat\alpha_{i,2s}z_iF^{(1)}_{sr}  &\text{if $i=2r-1, 2r$},\\
z_iF^{(1)}_{sr}  & \text{otherwise},
\end{cases}
\qquad 
F^{(2)}_{sr}z_i=
\begin{cases}
\widehat\alpha_{2s,i}z_iF^{(2)}_{sr}  &\text{if $i=2r-1, 2r$},\\
z_iF^{(2)}_{sr}  & \text{otherwise}.
\end{cases}
\]
\item Any two elements of $\{F^{(a)}_{sr} \mid a\in \{1,2\},\ 1\leq r<s\leq m \}$ commute with each other in $C(A_\al)$.
\item In particular, $F_{sr}^{(1)}F_{sr}^{(2)}=F_{sr}^{(2)}F_{sr}^{(1)}=1$ in $C(A_\al)$.
\end{enumerate}
\end{lem}

\begin{proof}
(1) This follows by $\widehat\alpha_{2r-1,i}\widehat\alpha_{2r,i}
=\al_{n,2r-1}\al_{n,2r}\al_{2r-1,i}\al_{2r,i}\al_{i, n}^2
\overset{\eqref{e.normalo}}{=}1$.

(2) The first result follows from
\begin{align*}
F^{(1)}_{sr}z_i &= (1-(1-\widehat\alpha_{2s,2r})z_{2r-1}z_{2r})z_i \\
&= 
\begin{cases}
(1-(1-\widehat\alpha_{2s,i+1})z_{i}z_{i+1})z_{i}
= \widehat\alpha_{2s,i+1}z_i \overset{(1)}{=} \widehat\alpha_{i,2s}z_i
&\text{if}\ i=2r-1, \\
(1-(1-\widehat\alpha_{2s,i})z_{i-1}z_i)z_i=z_i &\text{if}\ i=2r, \\
z_i -\widehat\alpha_{i,2r-1}\widehat\alpha_{i,2r}(1-\widehat\alpha_{2s,2r})
z_iz_{2r-1}z_{2r}
\overset{(1)}{=}
z_iF_{sr}
& \text{otherwise},
\end{cases}\\
z_i F^{(1)}_{sr} &=
\begin{cases}
z_i(1-(1-\widehat\alpha_{2s,i+1})z_{i}z_{i+1})
= z_i &\text{if}\ i=2r-1,\\
z_i(1-(1-\widehat\alpha_{2s,i})z_{i-1}z_{i}) = \widehat\alpha_{2s,i}z_i = \widehat\alpha_{i,2s}^{-1}z_i
&\text{if}\ i=2r.
\end{cases}
\end{align*}

The second result follows from
\begin{align*}
F^{(2)}_{sr}z_i &=
(1-(1-\widehat\alpha_{2r,2s})z_{2r-1}z_{2r})z_i\\
&=
\begin{cases}
(1 -(1-\widehat\alpha_{i+1,2s})z_{i}z_{i+1})z_{i}
=\widehat\alpha_{i+1,2s}z_i
\overset{(1)}{=}\widehat\alpha_{2s,i}z_i
&\text{if}\ i=2r-1, \\
(1-(1-\widehat\alpha_{i,2s})z_{i-1}z_i)z_i
=z_i &\text{if}\ i=2r, \\
z_i-\widehat\alpha_{i,2r-1}\widehat\alpha_{i,2r}(1-\widehat\alpha_{2r,2s})z_iz_{2r-1}z_{2r}
\overset{(1)}{=}
z_iF_{sr}
& \text{otherwise},
\end{cases}\\
z_i F^{(2)}_{sr} &=
\begin{cases}
z_i(1-(1-\widehat\alpha_{i+1,2s})z_{i}z_{i+1})
= z_i &\text{if}\ i=2r-1,\\
z_i(1-(1-\widehat\alpha_{i,2s})z_{i-1}z_{i})=
\widehat\alpha_{i,2s}z_i=\widehat\alpha_{2s,i}^{-1}z_i&\text{if}\ i=2r.
\end{cases}
\end{align*}

(3) Let $1\leq p<q \leq m$. By (2), we have
\[
F^{(1)}_{sr}z_{2p-1}z_{2p} =
\begin{cases}
z_{2p-1}z_{2p}-(1-\widehat\alpha_{2s,2p})z_{2p-1}z_{2p}z_{2p-1}z_{2p} = z_{2p-1}z_{2p}F^{(1)}_{sr}
& \text{if $p=r$},\\
z_{2p-1}z_{2p}F^{(1)}_{sr}  & \text{otherwise}.
\end{cases}
\]
Similarly, $F^{(2)}_{sr}z_{2p-1}z_{2p}=z_{2p-1}z_{2p}F^{(2)}_{sr}$. Therefore, we get $F^{(a)}_{sr}F^{(b)}_{qp}=F^{(b)}_{qp}F^{(a)}_{sr}$.

(4) We have
\begin{align*}
&F^{(1)}_{sr}F^{(2)}_{sr}=
(1-(1-\widehat\alpha_{2s, 2r})z_{2r-1}z_{2r} )
(1-(1-\widehat\alpha_{2r, 2s})z_{2r-1}z_{2r} )\\
&=1-(2-\widehat\alpha_{2s, 2r}-\widehat\alpha_{2r, 2s})z_{2r-1}z_{2r}+(2-\widehat\alpha_{2s, 2r}-\widehat\alpha_{2r, 2s})z_{2r-1}z_{2r}z_{2r-1}z_{2r}\\
&=1-(2-\widehat\alpha_{2s, 2r}-\widehat\alpha_{2r, 2s})z_{2r-1}z_{2r}+(2-\widehat\alpha_{2s, 2r}-\widehat\alpha_{2r, 2s})z_{2r-1}z_{2r}=1.\qedhere
\end{align*}
\end{proof} 

\begin{thm}\label{t.oiso}
$C(A_\al)$ is isomorphic to $C(A_{\textnormal{comm}})$.
\end{thm}

\begin{proof}
By Lemma \ref{l.o1}(4), $C(A_\al)$ is given by $k\langle z_1, \dots, z_{2m} \rangle$ with relations
\begin{align*}
z_iz_j+\widehat\al_{ji}z_jz_i \;\;(1\leq i<j \leq 2m, (i,j) \not\in L),\quad
z_iz_j+z_jz_i-1 \;\; ((i,j) \in L), \quad  z_i^2 \;\;(1\leq i \leq 2m),
\end{align*}
and $C(A_{\textnormal{comm}})$ is given by $k\langle z_1', \dots, z_{2m}' \rangle$ with relations
\begin{align*}
z_i'z_j'+z_j'z_i' \;\;(1\leq i<j \leq 2m, (i,j) \not\in L),\quad
z_i'z_j'+z_j'z_i'-1 \;\; ((i,j) \in L), \quad  (z_i')^2 \;\;(1\leq i \leq 2m).
\end{align*}
We define an algebra homomorphism $\psi: k\langle z_1', \dots, z_{2m}' \rangle \to C(A_\al)$ by 
\begin{align*}
\psi(z_{2s-1}') =z_{2s-1}\prod_{r=1}^{s-1}F^{(1)}_{sr}, \quad
\psi(z_{2s}') = z_{2s}\prod_{r=1}^{s-1}F^{(2)}_{sr}
\quad \ \text{for}\ 1\leq s\leq m.
\end{align*}

First, we show that this induces an algebra homomorphism $C(A_{\textnormal{comm}}) \to C(A_\al)$.
If $i=2s-1$ and $j=2t-1$ are odd with $1\leq s<t\leq m$, then
\begin{align*}
&\psi(z_i'z_j'+z_j'z_i')=
(z_{2s-1}\prod_{r=1}^{s-1}F^{(1)}_{sr})(z_{2t-1}\prod_{r=1}^{t-1}F^{(1)}_{tr})+(z_{2t-1}\prod_{r=1}^{t-1}F^{(1)}_{tr})(z_{2s-1}\prod_{r=1}^{s-1}F^{(1)}_{sr})
\\
&\overset{\text{Lem.\,\ref{l.o2}(2)}}{=}
z_{2s-1}z_{2t-1}(\prod_{r=1}^{s-1}F^{(1)}_{sr})(\prod_{r=1}^{t-1}F^{(1)}_{tr})+\widehat\alpha_{2s-1,2t}z_{2t-1}z_{2s-1}(\prod_{r=1}^{t-1}F^{(1)}_{tr})(\prod_{r=1}^{s-1}F^{(1)}_{sr})\\
&\overset{\text{Lem.\,\ref{l.o2}(3)}}{=}
(z_{2s-1}z_{2t-1}+\widehat\alpha_{2s-1,2t}z_{2t-1}z_{2s-1})(\prod_{r=1}^{s-1}F^{(1)}_{sr})(\prod_{r=1}^{t-1
}F^{(1)}_{tr})\\
&\overset{\text{Lem.\,\ref{l.o2}(1)}}{=}
(z_{i}z_{j}+\widehat\alpha_{ji}z_{j}z_{i})(\prod_{r=1}^{s-1}F^{(1)}_{sr})(\prod_{r=1}^{t-1
}F^{(1)}_{tr})
=0.
\end{align*}
If $i=2s-1$ is odd and $j=2t$ is even with $1\leq s<t\leq m$, then
\begin{align*}
&\psi(z_i'z_j'+z_j'z_i')
=(z_{2s-1}\prod_{r=1}^{s-1}F^{(1)}_{sr})(z_{2t}\prod_{r=1}^{t-1}F^{(2)}_{tr})+(z_{2t}\prod_{r=1}^{t-1}F^{(2)}_{tr})(z_{2s-1}\prod_{r=1}^{s-1}F^{(1)}_{sr})\\
&\overset{\text{Lem.\,\ref{l.o2}(2)}}{=}
z_{2s-1}z_{2t}(\prod_{r=1}^{s-1}F^{(1)}_{sr})(\prod_{r=1}^{t-1}F^{(2)}_{tr})+\widehat\alpha_{2t,2s-1}z_{2t}z_{2s-1}(\prod_{r=1}^{t-1}F^{(2)}_{tr})(\prod_{r=1}^{s-1}F^{(1)}_{sr})\\
&\overset{\text{Lem.\,\ref{l.o2}(3)}}{=}
(z_{i}z_{j}+\widehat\alpha_{ji}z_{j}z_{i})(\prod_{r=1}^{s-1}F^{(1)}_{sr})(\prod_{r=1}^{t-1}F^{(2)}_{tr})=0.
\end{align*}
Similarly, one can check that if $1\leq i<j\leq 2m$, 
and $i$ is even, then
$\psi(z_i'z_j'+z_j'z_i')=0$.
Moreover, for $1\leq s \leq m$, we have
\begin{align*}
\psi(z'_{2s-1}z'_{2s}+z'_{2s}z'_{2s-1}-1) 
&= 
(z_{2s-1}\prod_{r=1}^{s-1}F^{(1)}_{sr})
(z_{2s}\prod_{r=1}^{s-1}F^{(2)}_{sr})+    
(z_{2s}\prod_{r=1}^{s-1}F^{(2)}_{sr})
(z_{2s-1}\prod_{r=1}^{s-1}F^{(1)}_{sr})-1\\
&\overset{\text{Lem.\,\ref{l.o2}(2),(3)}}{=}
z_{2s-1}z_{2s}(\prod_{r=1}^{s-1}F^{(1)}_{sr}F^{(2)}_{sr})+
z_{2s}z_{2s-1}(\prod_{r=1}^{s-1}F^{(2)}_{sr}F^{(1)}_{sr})
-1\\
&\overset{\text{Lem.\,\ref{l.o2}(4)}}{=}
z_{2s-1}z_{2s}+z_{2s}z_{2s-1}-1=0.
\end{align*}
For $1\leq s \leq m$, we have
\[
\psi((z_{2s-1}')^2)= (z_{2s-1}\prod_{r=1}^{s-1}F^{(1)}_{sr})^2\overset{\text{Lem.\,\ref{l.o2}(2)}}{=}z_{2s-1}^2(\prod_{r=1}^{s-1}F^{(1)}_{sr})^2=0,
\]
and similarly $\psi((z_{2s}')^2)=0$.
Thus we get the induced algebra homomorphism $\overline{\psi}: C(A_{\textnormal{comm}}) \to C(A_\al)$.

Next, we show by induction that each $z_i$ lies in $\Im \overline{\psi}$, which implies that $\overline{\psi}$ is surjective.
Clearly, $z_1 = \overline{\psi}(z_1'), z_2 = \overline{\psi}(z_2') \in \Im \overline{\psi}$. Suppose that $z_i \in \Im \overline{\psi}$ for all $1\leq i\leq 2\ell-2$. Then $F^{(1)}_{\ell r}, F^{(2)}_{\ell r} \in \Im \overline{\psi}$ for all $1\leq r \leq \ell-1$, so 
\begin{align*}
&z_{2\ell-1} 
\overset{\text{Lem.\,\ref{l.o2}(4)}}{=}z_{2\ell-1}(\prod_{r=1}^{\ell-1}F^{(1)}_{\ell r}F^{(2)}_{\ell r})
\overset{\text{Lem.\,\ref{l.o2}(3)}}{=} \overline{\psi}(z_{2\ell-1}')(\prod_{r=1}^{\ell-1}F^{(2)}_{\ell r}) \in \Im \overline{\psi},\\
&z_{2\ell}
\overset{\text{Lem.\,\ref{l.o2}(4)}}{=} z_{2\ell}(\prod_{r=1}^{\ell-1}F^{(2)}_{\ell r}F^{(1)}_{\ell r})
\overset{\text{Lem.\,\ref{l.o2}(3)}}{=}\overline{\psi}(z_{2\ell}')(\prod_{r=1}^{\ell-1}F^{(1)}_{\ell r}) \in \Im \overline{\psi}.
\end{align*}
Therefore, $\overline{\psi}$ is surjective.

Since $\dim_k C(A_{\textnormal{comm}})= \dim_k C(A_\al) =2^{2m}$, it follows that 
$\overline{\psi}$ is an isomorphism.
\end{proof}

We now prove Theorem \ref{t.intromain} in the case where $n$ is odd.

\begin{proof}[Proof of Theorem \ref{t.intromain} for odd $n$]

Since $A_{\textnormal{comm}} \cong k[x_1,\dots,x_n]/(x_1^2+\dots+x_n^2)=:B$,
it follows from Theorem \ref{t.oiso} and \eqref{e.commCA} that $C(A_\al) \cong C(A_{\textnormal{comm}}) \cong C(B) \cong M_{2^{(n-1)/2}}(k)$.
Furthermore, Theorem \ref{t.SV} and Morita theory imply that
$\uCM(A_\al) \simeq \D^b(\mod  M_{2^{(n-1)/2}}(k)^{\op})\simeq \D^b(\mod k)$.
\end{proof}

\section{Proof of Theorem \ref{t.intromain}: the case of an even number of variables}\label{sec.even}

This section is devoted to proving Theorem \ref{t.intromain} in the case where the number of variables is even. 
While the even-variable case requires more technical arguments, the overall strategy is the same as in the odd-variable case. We therefore adopt the same notation for the analogous objects, despite the differences in their definitions.
Throughout this section,

\begin{itemize}
\item $n:=2m+2$ with $m \geq 0$,
\item $S_\al = k\langle x_1, \dots, x_{n} \rangle /(x_ix_j -\al_{ij} x_jx_i)$ is a standard graded skew polynomial algebra in $n=2m+2$ variables, 
\item $f:=x_1x_2+\cdots+x_{2m-1}x_{2m}+x_{n-1}x_{n} \in S_\al $,
\item $A_\al :=S_\al /(f)$,
\item $A_{\textnormal{comm}}:=k[x_1,\dots,x_n]/(x_1x_2+\cdots+x_{2m-1}x_{2m}+x_{n-1}x_{n})$,
that is, this is the special case of $S_\al /(f)$ where $\al_{ij}=1$ for all $1\leq i,j\leq n$.
\end{itemize}

\begin{lem}
The element $f \in S_\al$ is normal if and only if 
\begin{align}\label{e.normale}
\al_{2t-1,2s}\al_{2t,2s}=\al_{2s-1,2t-1}\al_{2s-1,2t}=\al_{2s-1,2s}
\end{align}
for all $1\leq s,t\leq m+1$ with $s\neq t$.
\end{lem}

\begin{proof}
$(\Rightarrow)$ Suppose that $f$ is normal. 
By the same argument as in the proof of Lemma \ref{l.normalo},
for each $1\leq i\leq n$, there exists $\la_{i}\in k$ such that $x_if=\la_{i}fx_i$.
Moreover, it is easily checked that
\begin{align}\label{e.normalcale}
\begin{split}
\la_{2s-1}=\al_{2s-1,2t-1}\al_{2s-1,2t}=\al_{2s-1,2s},\qquad 
\la_{2s}=\al_{2s,2t-1}\al_{2s,2t}=\al_{2s,2s-1}
\end{split}
\end{align}
for all $1\leq s,t\leq m+1$ with $s\neq t$. Thus we have \eqref{e.normale}.

$(\Leftarrow)$ Suppose that $\eqref{e.normale}$. Define $\la_{2s-1}, \la_{2s}$ (for $1\leq s\leq m+1$) as in \eqref{e.normalcale}.
Then one can check that $x_if=\la_{i}fx_i$ for all $1 \leq i \leq  2m+2$, so $f$ is a normal element.
\end{proof}

\begin{lem} \label{l.e1}
Assume that $f \in S_{\al}$ is normal. Let $L'=\{(2s-1,2s) \mid 1\leq s \leq m+1 \}$.
\begin{enumerate}
\item $S_\al^!$ is isomorphic to $k\langle x_1, \dots, x_{n} \rangle$ with relations
\begin{align*}
x_ix_j+\al_{ji}x_jx_i \;\; (1\leq i<j \leq n), \quad x_i^2 \;\; (1\leq i \leq n).
\end{align*}
\item $A_\al^!$ is isomorphic to $k\langle x_1, \dots, x_{n} \rangle$ with relations
\begin{align*}
&x_ix_j+\al_{ji}x_jx_i \;\;(1\leq i<j \leq n,\ (i,j) \not\in L'),\\ 
&x_{i}x_{j}+\al_{ji}x_{j}x_{i}-(x_{n-1}x_{n}+\al_{n,n-1}x_{n}x_{n-1}) \;\; ((i,j) \in L'), \quad 
x_i^2 \;\;(1\leq i \leq n).
\end{align*}
\item $w:=x_{n-1}x_{n}+\al_{n,n-1}x_{n}x_{n-1}\in A_\al ^!$ is a normal element such that $A_\al^!/(w) \cong S_\al^!$.
\item
$C(A_\al)=A_\al^![w^{-1}]_0$ is isomorphic to $k\langle y_1, \dots, y_{2m},z_1, \dots, z_{2m} \rangle$ with relations
\begin{align}\label{e.C(A)e}
\begin{split}
&y_i z_i,\ z_i y_i \;\; (1\leq i\leq 2m),\quad 
y_i y_j,\ z_i z_j\;\; (1\leq i,j\leq 2m), \\
&y_{2s-1}z_{2t-1}+\al_{2t-1,2s-1}y_{2t-1}z_{2s-1},\quad 
z_{2s-1}y_{2t-1}+\al_{2t,2s}z_{2t-1}y_{2s-1}\;\;(1\leq s<t\leq m),
\\
&y_{2s}z_{2t}+\al_{2t-1,2s-1}y_{2t}z_{2s},\quad 
z_{2s}y_{2t}+\al_{2t,2s}z_{2t}y_{2s}\;\; (1\leq s<t\leq m),
\\
&y_{2s-1}z_{2t}+\al_{2s,2t}y_{2t}z_{2s-1},\quad
z_{2s-1}y_{2t}+\al_{2s-1,2t-1}z_{2t}y_{2s-1} \;\; (1\leq s,t\leq m,\ s\neq t),
\\
&y_{2s-1}z_{2s}+y_{2s}z_{2s-1}-(y_{1}z_{2}+y_{2}z_{1}), \quad
z_{2s-1}y_{2s}+z_{2s}y_{2s-1}-(z_{1}y_{2}+z_{2}y_{1}) \;\; (1\leq s\leq m),
\\
&y_{1}z_{2}+y_{2}z_{1}+z_{1}y_{2}+z_{2}y_{1}-1.
\end{split}
\end{align}
\end{enumerate}
\end{lem}

\begin{proof}
(1) and (2) follow from a direct computation.

(3) Since $x_{2s-1}w=\al_{2s,2s-1}wx_{2s-1}$
and  $x_{2s}w=\al_{2s-1,2s}wx_{2s}$
for $1\leq s\leq m+1$, we see that $w$ is normal.
The last isomorphism immediately follows from (1) and (2).

(4) We define an algebra homomorphism $\vph: k\langle y_1, \dots, y_{2m}, z_1, \dots, z_{2m} \rangle \to C(A_\al)$ by 
\begin{align*}
&y_{2s-1} \mapsto -\al_{n,2s-1}x_{2s-1}x_{n-1}w^{-1},
&&y_{2s}\mapsto -\al_{2s, n-1}x_{2s}x_{n-1}w^{-1},\\
&z_{2s-1} \mapsto x_{2s-1}x_{n}w^{-1}, 
&&z_{2s} \mapsto x_{2s}x_{n}w^{-1} 
&& \quad\text{for $1\leq s \leq m$.}
\end{align*}

Since $w^{-1}x_{2s-1}=\al_{2s,2s-1}x_{2s-1}w^{-1}$ and $w^{-1}x_{2s}=\al_{2s-1,2s}x_{2s}w^{-1}$ for $1\leq s \leq m+1$, it follows that
$x_{i}x_{j}w^{-1}$ is a linear combination of
$x_{i}x_{n-1}w^{-1}x_{j}x_{n}w^{-1}$ and $x_{i}x_{n}w^{-1}x_{j}x_{n-1}w^{-1}$
for any $1\leq i,j \leq 2m$.
This implies that $\{x_ix_{n-1}w^{-1}, x_{i}x_{n}w^{-1} \mid 1\leq i \leq2m \}$ generates $C(A_\al)$, and hence $\vph$ is surjective. 

It is easy to check that $\vph(y_iz_i)=\vph(z_iy_i)=\vph(
y_iy_j)=\vph(z_iz_j)=0$.
Moreover, we have
\begin{align*}
&\vph(y_{2s-1}z_{2t-1}+\al_{2t-1,2s-1}y_{2t-1}z_{2s-1})\\
&=-\al_{n,2s-1}x_{2s-1}x_{n-1}w^{-1}x_{2t-1}x_{n}w^{-1}
-\al_{2t-1,2s-1}\al_{n,2t-1}x_{2t-1}x_{n-1}w^{-1}x_{2s-1}x_{n}w^{-1}\\
&\overset{\eqref{e.normale}}{=}\al_{n,2s-1}\al_{n,2t-1}(x_{2s-1}x_{2t-1}x_{n-1}w^{-1}x_{n}w^{-1}
+\al_{2t-1,2s-1}x_{2t-1}x_{2s-1}x_{n-1}w^{-1}x_{n}w^{-1})\\
&=\al_{n,2s-1}\al_{n,2t-1}(1-\al_{2t-1,2s-1}\al_{2s-1,2t-1})
x_{2s-1}x_{2t-1}x_{n-1}w^{-1}x_{n}w^{-1}=0,
\end{align*}
\begin{align*}
&\vph(y_{2s}z_{2t}+\al_{2t-1,2s-1}y_{2t}z_{2s})\\
&=-\al_{2s, n-1}x_{2s}x_{n-1}w^{-1}x_{2t}x_{n}w^{-1}
-\al_{2t-1,2s-1}\al_{2t, n-1}x_{2t}x_{n-1}w^{-1}x_{2s}x_{n}w^{-1}\\
&\overset{\eqref{e.normale}}{=}
\al_{2s, n-1}\al_{2t, n-1}(\al_{2t-1,2t}x_{2s}x_{2t}x_{n-1}w^{-1}x_{n}w^{-1}
+\al_{2s-1,2t}x_{2t}x_{2s}x_{n-1}w^{-1}x_{n}w^{-1})\\
&=
\al_{2s, n-1}\al_{2t, n-1}(\al_{2t-1,2t}
-\al_{2s-1,2t}\al_{2s,2t})x_{2s}x_{2t}x_{n-1}w^{-1}x_{n}w^{-1}\overset{\eqref{e.normale}}{=}0,
\end{align*}
\begin{align*}
&\vph(y_{2s-1}z_{2t}+\al_{2s,2t}y_{2t}z_{2s-1})\\
&=-\al_{n,2s-1}x_{2s-1}x_{n-1}w^{-1}x_{2t}x_{n}w^{-1}
-\al_{2s,2t}\al_{2t, n-1}x_{2t}x_{n-1}w^{-1}x_{2s-1}x_{n}w^{-1}\\
&\overset{\eqref{e.normale}}{=}\al_{2t,n-1}\al_{n,2s-1}(\al_{2t-1,2t}x_{2s-1}x_{2t}x_{n-1}w^{-1}x_{n}w^{-1}
+\al_{2s,2t}x_{2t}x_{2s-1}x_{n-1}w^{-1}x_{n}w^{-1})\\
&=\al_{2t,n-1}\al_{n,2s-1}(\al_{2t-1,2t}
-\al_{2s,2t}\al_{2s-1,2t})x_{2s-1}x_{2t}x_{n-1}w^{-1}x_{n}w^{-1}\overset{\eqref{e.normale}}{=}0.
\end{align*}
Similarly, one can verify that $\vph(z_{2s-1}y_{2t-1}+\al_{2t,2s}z_{2t-1}y_{2s-1})=\vph(z_{2s}y_{2t}+\al_{2t,2s}z_{2t}y_{2s})=\vph(z_{2s-1}y_{2t}+\al_{2s-1,2t-1}z_{2t}y_{2s-1})=0$.

Since $w=x_{2s-1}x_{2s}+\al_{2s,2s-1}x_{2s}x_{2s-1}$ for every $1\leq s \leq m$, we obtain
\begin{align*}
&x_{n-1}x_{n}w^{-1}=(x_{2s-1}x_{2s}+\al_{2s,2s-1}x_{2s}x_{2s-1})w^{-1}x_{n-1}x_{n}w^{-1}\\
&=-\al_{n,n-1}\al_{n-1,2s}\al_{2s,2s-1}x_{2s-1}x_{n-1}w^{-1}x_{2s}x_{n}w^{-1}
-\al_{n,n-1}\al_{n-1,2s-1}x_{2s}x_{n-1}w^{-1}x_{2s-1}x_{n}w^{-1}
\\
&\overset{\eqref{e.normale}}{=}-\al_{n,2s-1}x_{2s-1}x_{n-1}w^{-1}x_{2s}x_{n}w^{-1}
-\al_{2s,n-1}x_{2s}x_{n-1}w^{-1}x_{2s-1}x_{n}w^{-1}
=y_{2s-1}z_{2s}+y_{2s}z_{2s-1},
\end{align*}
and similarly
$\al_{n,n-1}x_{n}x_{n-1}w^{-1}=z_{2s-1}y_{2s}+z_{2s}y_{2s-1}$.
These show $\vph(y_{2s-1}z_{2s}+y_{2s}z_{2s-1}-(y_{1}z_{2}+y_{2}z_{1}))=\vph(z_{2s-1}y_{2s}+z_{2s}y_{2s-1}-(z_{1}y_{2}+z_{2}y_{1}))=\vph(y_{1}z_{2}+y_{2}z_{1}+z_{1}y_{2}+z_{2}y_{1}-1)=0$.

Therefore, the relations in \eqref{e.C(A)e} lie in $\Ker \vph$. This yields that $\vph$ induces a surjective algebra homomorphism from $k\langle y_1, \dots, y_{2m}, z_1, \dots, z_{2m}\rangle$ subject to the relations \eqref{e.C(A)e} to $C(A_\al)$.

The algebra $k\langle y_1, \dots, y_{2m}, z_1, \dots, z_{2m}\rangle$ subject to the relations \eqref{e.C(A)e} has a $k$-basis consisting of
\[
\{1,\ y_2z_1,\ y_{i_1}z_{i_2}y_{i_3}\cdots u_{i_s},\ z_{i_1}y_{i_2}z_{i_3}\cdots u'_{i_s} \mid 1\leq i_1<\cdots<i_s\leq 2m,\ 1\leq s\leq 2m\},
\]
where the terminal symbols $u$ and $u'$ are determined by the parity of $s$ as above.
Therefore, the dimension of this algebra is $2^{2m+1}$.
Since  $\dim_k C(A_\al)=2^{2m+1}$ by Proposition \ref{p.dim}, the induced homomorphism is an isomorphism.
\end{proof}

For the remainder of this section, assume that 
$f \in S_\al$ is normal, and use the following notation.
\begin{itemize}
\item For $1\leq s\leq m$, fix $\be_{2s-1,2s} \in k$ such that $\be_{2s-1,2s}^2 = \al_{2s-1,2s}$ and define $\be_{2s,2s-1} \in k$ by $\be_{2s,2s-1}= \be_{2s-1,2s}^{-1}$. (Then $\be_{2s,2s-1}^2=\al_{2s,2s-1}$.) 
\item For $1\leq i,j\leq 2m$, define $\widehat\alpha_{ij}=
\begin{cases}
\be_{i+1,i}\al_{ij}\be_{j,j+1} &\text{if $i,j$ are odd},\\
\be_{i+1,i}\al_{ij}\be_{j,j-1} &\text{if $i$ is odd and $j$ is even},\\
\be_{i-1,i}\al_{ij}\be_{j,j+1} &\text{if $i$ is even and $j$ is odd},\\
\be_{i-1,i}\al_{ij}\be_{j,j-1} &\text{if $i, j$ are even}.
\end{cases}$
\item For $1\leq r<s \leq m$, define the elements $F^{(1)}_{sr}, F^{(2)}_{sr}, G^{(1)}_{sr}, G^{(2)}_{sr} \in C(A_\al)$ by
\begin{align*}
&F_{sr}^{(1)} = 1-(1-\widehat\alpha_{2s,2r})z_{2r-1}y_{2r}, 
&&F_{sr}^{(2)} = 1-(1-\widehat\alpha_{2s,2r})z_{2r}y_{2r-1},\\
&G_{sr}^{(1)}= 1-(1-\widehat\alpha_{2s,2r})y_{2r-1}z_{2r}, 
&&G_{sr}^{(2)} = 1-(1-\widehat\alpha_{2s,2r})y_{2r}z_{2r-1}, 
\end{align*}
and define the elements $E, E' \in C(A_\al)$ by
\[E= z_1y_2+z_2y_1, \qquad\qquad  E'= y_1z_2+y_2z_1, \]
where we identify the isomorphism obtained in Lemma \ref{l.e1}(4). 
\end{itemize}
Note that $E$ and $E'$ are idempotents of $C(A_\al)$ satisfying
\begin{align}\label{e.ip}
E=z_{2s-1}y_{2s}+z_{2s}y_{2s-1},\quad 
E'=y_{2s-1}z_{2s}+y_{2s}z_{2s-1},\quad
E+E'=1
\end{align}
for any $1 \leq s\leq m$ by \eqref{e.C(A)e}.

\begin{lem}\label{l.e2}
Assume that $f \in S_\al$ is normal.
\begin{enumerate}
\item $\widehat\alpha_{2r-1,i}\widehat\alpha_{2r,i}=\widehat\alpha_{i,2r-1}\widehat\alpha_{i,2r}= 1$ for $1\leq r \leq m$ and $1\leq i \leq 2m$.
\item In $C(A_{\al})$, for $a \in\{1,2\}$, we have
\begin{align*}
F_{sr}^{(1)}z_i=
\begin{cases}
\widehat\alpha_{i,2s}z_{i} G_{sr}^{(1)} & \text{if $i=2r-1,2r$,}\\
z_i G_{sr}^{(1)} &\text{otherwise,}
\end{cases}
\qquad 
F_{sr}^{(2)}z_i=
\begin{cases}
\widehat\alpha_{2s,i}z_{i} G_{sr}^{(2)} & \text{if $i=2r-1,2r$,}\\
z_i G_{sr}^{(2)} &\text{otherwise,}
\end{cases}\\
G_{sr}^{(1)}y_{i}=
\begin{cases}
\widehat\alpha_{i,2s}y_{i} F_{sr}^{(1)}&\text{if $i=2r-1,2r$,}\\
y_{i} F_{sr}^{(1)}&\text{otherwise,}
\end{cases}
\qquad
G_{sr}^{(2)}y_{i}=
\begin{cases}
\widehat\alpha_{2s,i}y_{i} F_{sr}^{(2)}&\text{if $i=2r-1,2r$,}\\
y_{i} F_{sr}^{(2)}&\text{otherwise.}
\end{cases}
\end{align*}
\item Any two elements of $\{F^{(a)}_{sr}, G^{(a)}_{sr}\mid a \in\{1,2\},\ 1\leq r<s\leq m\} $ commute with each other in $C(A_{\al})$.
\item In particular, one has
$F^{(1)}_{sr}F^{(2)}_{sr}=F^{(2)}_{sr}F^{(1)}_{sr}=1-(1-\widehat\alpha_{2s,2r})E$ and 
$G^{(1)}_{sr}G^{(2)}_{sr}=G^{(2)}_{sr}G^{(1)}_{sr}=1-(1-\widehat\alpha_{2s,2r})E'$ in $C(A_{\al})$.
\end{enumerate}
\end{lem}

\begin{proof}
(1) We have
\[
\widehat\alpha_{2r-1,i}\widehat\alpha_{2r,i}
=
\begin{cases}
\be_{2r,2r-1}\al_{2r-1,i}\be_{i,i+1}\be_{2r-1,2r}\al_{2r,i}\be_{i,i+1}
=\al_{2r-1,i}\al_{2r,i}\al_{i,i+1}\overset{\eqref{e.normale}}{=}1
& \text{if $i$ is odd},\\
\be_{2r,2r-1}\al_{2r-1,i}\be_{i,i-1}\be_{2r-1,2r}\al_{2r,i}\be_{i,i-1}
=\al_{2r-1,i}\al_{2r,i}\al_{i,i-1}\overset{\eqref{e.normale}}{=}1
&\text{if $i$ is even}.
\end{cases}
\]

(2) \underline{The cases $i \not\in \{2r-1,2r\}$.}
To prove $F_{sr}^{(1)}z_i=z_i G_{sr}^{(1)}$, it suffices to show that $z_{2r-1}y_{2r}z_{i}=z_{i}y_{2r-1}z_{2r}$.
If $i$ is odd, then $z_{2r-1}y_{2r}z_{i}=-\al_{2r,i+1}z_{2r-1}y_{i}z_{2r}=z_{i}y_{2r-1}z_{2r}$.
If $i$ is even, then $z_{2r-1}y_{2r}z_{i}=-\al_{i-1,2r-1}z_{2r-1}y_{i}z_{2r}=z_{2r-1}y_{i}z_{2r}$. Thus the claim holds.

To prove $F_{sr}^{(2)}z_i =z_i G_{sr}^{(2)}$, it suffices to show that
$z_{2r}y_{2r-1}z_{i}=z_{i}y_{2r}z_{2r-1}$.
If $i$ is odd, then $z_{2r}y_{2r-1}z_{i}=-\al_{i,2r-1}z_{2r}y_{i}z_{2r-1}=z_{i}y_{2r}z_{2r-1}$.
If $i$ is even, then $z_{2r}y_{2r-1}z_{i}=-\al_{2r,i}z_{2r}y_{i}z_{2r-1}=z_{i}y_{2r}z_{2r-1}$. Thus the claim holds.

The last two cases are verified in a similar manner.

\underline{The cases $i \in \{2r-1,2r\}$.} \eqref{e.C(A)e} implies $y_{2r-1}z_{2r}+y_{2r}z_{2r-1}+
z_{2r-1}y_{2r}+z_{2r}y_{2r-1}=1$
for any $1\leq r\leq m$. By multiplying this equality on the left by $y_{2r-1},y_{2r},z_{2r-1},z_{2r}$, respectively, we obtain
\begin{align*}
&y_{2r-1}z_{2r}y_{2r-1}=y_{2r-1}, &&y_{2r}z_{2r-1}y_{2r}=y_{2r},
&z_{2r-1}y_{2r}z_{2r-1}=z_{2r-1}, &&z_{2r}y_{2r-1}z_{2r}=z_{2r}.
\end{align*}
Using these equalities, we get
\begin{align*}
F_{sr}^{(1)}z_{2r-1}
&=z_{2r-1}-(1-\widehat\alpha_{2s,2r})z_{2r-1}y_{2r}z_{2r-1}
=\widehat\alpha_{2s,2r}z_{2r-1}\\
&=\widehat\alpha_{2s,2r}(z_{2r-1}-(1-\widehat\alpha_{2s,2r})z_{2r-1}y_{2r-1}z_{2r})
\overset{(1)}{=}\widehat\alpha_{2r-1,2s}z_{2r-1}G_{sr}^{(1)},\\
F_{sr}^{(1)}z_{2r} &=z_{2r}-(1-\widehat\alpha_{2s,2r})z_{2r-1}y_{2r}z_{2r}
=z_{2r}\\
&=\widehat\alpha_{2s,2r}^{-1}(z_{2r}-(1-\widehat\alpha_{2s,2r})z_{2r}y_{2r-1}z_{2r})
=\widehat\alpha_{2r,2s}z_{2r}G_{sr}^{(1)},\\
F_{sr}^{(2)}z_{2r-1}
&=z_{2r-1}-(1-\widehat\alpha_{2s,2r})z_{2r}y_{2r-1}z_{2r-1}
=z_{2r-1}\\
&=\widehat\alpha_{2s,2r}^{-1}(z_{2r-1}-(1-\widehat\alpha_{2s,2r})z_{2r-1}y_{2r}z_{2r-1})
=\widehat\alpha_{2s,2r}^{-1}z_{2r-1}G_{sr}^{(2)}
\overset{(1)}{=}\widehat\alpha_{2s,2r-1}z_{2r-1}G_{sr}^{(2)},\\
F_{sr}^{(2)}z_{2r} &=z_{2r}-(1-\widehat\alpha_{2s,2r})z_{2r}y_{2r-1}z_{2r}
=\widehat\alpha_{2s,2r}z_{2r}\\
&=\widehat\alpha_{2s,2r}(z_{2r}-(1-\widehat\alpha_{2s,2r})z_{2r}y_{2r}z_{2r-1})
=\widehat\alpha_{2s,2r}z_{2r}G_{sr}^{(2)}.
\end{align*}
The remaining cases are treated similarly.

(3) Since $y_iy_j=z_iz_j=0$, we see easily that $F_{sr}^{(a)}G_{qp}^{(b)}=G_{qp}^{(b)}F_{sr}^{(a)}$.
Since
\[
F_{sr}^{(a)}z_{2p-1}y_{2p}=z_{2p-1}y_{2p}F_{sr}^{(a)},\qquad
F_{sr}^{(a)}z_{2p}y_{2p-1}=z_{2p}y_{2p-1}F_{sr}^{(a)}
\]
are obtained from (2) in both cases $p \neq r$ and $p = r$, it follows that $F_{sr}^{(a)}F_{qp}^{(b)}=F_{qp}^{(b)}F_{sr}^{(a)}$.
The equality $G_{sr}^{(a)}G_{qp}^{(b)}=G_{qp}^{(b)}G_{sr}^{(a)}$ can be proved in the same way.

(4) By \eqref{e.ip}, we have
$F_{sr}^{(1)}F_{sr}^{(2)}
=1-(1-\widehat\alpha_{2s,2r})(z_{2r-1}y_{2r}+z_{2r}y_{2r-1})=1-(1-\widehat\alpha_{2s,2r})E$ and $G_{sr}^{(1)}G_{sr}^{(2)}
=1-(1-\widehat\alpha_{2s,2r})(y_{2r-1}z_{2r}+y_{2r}z_{2r-1})=1-(1-\widehat\alpha_{2s,2r})E'$. Hence the result.
\end{proof}

\begin{thm}\label{t.eiso}
$C(A_\al)$ is isomorphic to $C(A_{\textnormal{comm}})$.
\end{thm}

\begin{proof}
By Lemma \ref{l.e1}(4), $C(A_\al)$ is given by $k\langle y_1, \dots, y_{2m}, z_1, \dots, z_{2m}\rangle$ with relations
\begin{align*}
&y_i z_i,\ z_i y_i \;\; (1\leq i\leq 2m),\quad 
y_i y_j,\ z_i z_j\;\; (1\leq i,j\leq 2m), \\
&y_{2s-1}z_{2t-1}+\al_{2t-1,2s-1}y_{2t-1}z_{2s-1},\quad 
z_{2s-1}y_{2t-1}+\al_{2t,2s}z_{2t-1}y_{2s-1}\;\;(1\leq s<t\leq m),
\\
&y_{2s}z_{2t}+\al_{2t-1,2s-1}y_{2t}z_{2s},\quad 
z_{2s}y_{2t}+\al_{2t,2s}z_{2t}y_{2s}\;\; (1\leq s<t\leq m),
\\
&y_{2s-1}z_{2t}+\al_{2s,2t}y_{2t}z_{2s-1},\quad
z_{2s-1}y_{2t}+\al_{2s-1,2t-1}z_{2t}y_{2s-1} \;\; (1\leq s,t\leq m,\ s\neq t),
\\
&y_{2s-1}z_{2s}+y_{2s}z_{2s-1}-(y_{1}z_{2}+y_{2}z_{1}), \quad
z_{2s-1}y_{2s}+z_{2s}y_{2s-1}-(z_{1}y_{2}+z_{2}y_{1}) \;\; (1\leq s\leq m),
\\
&y_{1}z_{2}+y_{2}z_{1}+z_{1}y_{2}+z_{2}y_{1}-1
\end{align*} 
and $C(A_{\textnormal{comm}})$ is given by $k\langle  y_1', \dots, y_{2m}', z_1', \dots, z_{2m}' \rangle$ with relations
\begin{align*}
&y_i'z_i',\ z_i'y_i' \;\; (1\leq i\leq 2m),\quad 
y_i'y_j',\ z_i'z_j'\;\; (1\leq i,j\leq 2m), \quad
y_{i}'z_{j}'+y_{j}'z_{i}',\ z_{i}'y_{j}'+z_{j}'y_{i}'\;\;(1\leq i<j\leq 2m),
\\
&y_{2s-1}'z_{2s}'+y_{2s}'z_{2s-1}'-(y_{1}'z_{2}'+y_{2}'z_{1}'), \quad
z_{2s-1}'y_{2s}'+z_{2s}'y_{2s-1}'-(z_{1}'y_{2}'+z_{2}'y_{1}') \;\; (1\leq s\leq m),\\
&y_{1}'z_{2}'+y_{2}'z_{1}'+z_{1}'y_{2}'+z_{2}'y_{1}'-1.
\end{align*} 
We define an algebra homomorphism $\psi: k\langle y_1', \dots, y_{2m}', z_1', \dots, z_{2m}' \rangle \to C(A_\al)$ by 
\begin{align*}
&\psi(y_{2s-1}') = y_{2s-1}\prod_{r=1}^{s-1}F^{(1)}_{sr}, &&\psi(y_{2s}')=
\be_{2s-1,2s}y_{2s}\prod_{r=1}^{s-1}\widehat\alpha_{2r,2s}F_{sr}^{(2)},\\
&\psi(z_{2s-1}')= \be_{2s,2s-1}z_{2s-1}\prod_{r=1}^{s-1}G_{sr}^{(1)},
&&\psi(z_{2s}')= z_{2s}\prod_{r=1}^{s-1}\widehat\alpha_{2r,2s}G_{sr}^{(2)} 
\end{align*}
for $1\leq s\leq m$.

First, we show that this induces an algebra homomorphism $C(A_{\textnormal{comm}}) \to C(A_\al)$.
It is straightforward to see that $\psi(y_i'y'_j)=\psi(z_i'z'_j)=0$.
Furthermore, by Lemma \ref{l.e2}(2), $\psi(y_i'z'_i)=\psi(z_i'y'_i)=0$.

We now calculate $\psi(y_{i}'z_{j}'+y_{j}'z_{i}')$.
If $i=2s-1, j=2t-1$ are odd with $1\leq s<t\leq m$, then
\begin{align*}
&\psi(y_{i}'z_{j}'+y_{j}'z_{i}')
=(y_{2s-1}\prod_{r=1}^{s-1}F^{(1)}_{sr})
(\be_{2t,2t-1}z_{2t-1}\prod_{r=1}^{t-1}G_{tr}^{(1)})+(y_{2t-1}\prod_{r=1}^{t-1}F^{(1)}_{tr})(\be_{2s,2s-1}z_{2s-1}\prod_{r=1}^{s-1}G_{sr}^{(1)})\\
&\overset{\text{Lem.\,\ref{l.e2}(2)}}{=}
\be_{2t,2t-1}y_{2s-1}z_{2t-1}(\prod_{r=1}^{s-1}G^{(1)}_{sr})
(\prod_{r=1}^{t-1}G_{tr}^{(1)})+\be_{2s,2s-1}\widehat\alpha_{2s-1,2t}y_{2t-1}z_{2s-1}(\prod_{r=1}^{t-1}G^{(1)}_{tr})(\prod_{r=1}^{s-1}G_{sr}^{(1)})\\
&\overset{\text{Lem.\,\ref{l.e2}(3)}}{=}
\be_{2t,2t-1}(y_{2s-1}z_{2t-1}(\prod_{r=1}^{s-1}G^{(1)}_{sr})
(\prod_{r=1}^{t-1}G_{tr}^{(1)})+\alpha_{2s,2s-1}\alpha_{2s-1,2t}y_{2t-1}z_{2s-1}(\prod_{r=1}^{s-1}G_{sr}^{(1)})(\prod_{r=1}^{t-1}G^{(1)}_{tr}))\\
&\overset{\eqref{e.normale}}{=}
\be_{2t-1,2t}(y_{2s-1}z_{2t-1}+\alpha_{2t-1,2s-1}y_{2t-1}z_{2s-1})(\prod_{r=1}^{s-1}G_{sr}^{(1)})(\prod_{r=1}^{t-1}G^{(1)}_{tr})=0.
\end{align*}
If $i=2s, j=2t$ are even with $1\leq s<t\leq m$, then
\begin{align*}
&\psi(y_{i}'z_{j}'+y_{j}'z_{i}')\\
&=(\be_{2s-1,2s}y_{2s}\prod_{r=1}^{s-1}\widehat\alpha_{2r,2s}F^{(2)}_{sr})
(z_{2t}\prod_{r=1}^{t-1}\widehat\alpha_{2r,2t}G_{tr}^{(2)})+(\be_{2t-1,2t}y_{2t}\prod_{r=1}^{t-1}\widehat\alpha_{2r,2t}F^{(2)}_{tr})
(z_{2s}\prod_{r=1}^{s-1}\widehat\alpha_{2r,2s}G_{sr}^{(2)})\\
&\overset{\text{Lem.\,\ref{l.e2}(2)}}{=}
\be_{2s-1,2s}y_{2s}z_{2t}(\prod_{r=1}^{s-1}\widehat\alpha_{2r,2s}G^{(2)}_{sr})
(\prod_{r=1}^{t-1}\widehat\alpha_{2r,2t}G_{tr}^{(2)})
+
\be_{2t-1,2t}\widehat\alpha_{2t,2s}y_{2t}z_{2s}(\prod_{r=1}^{t-1}\widehat\alpha_{2r,2t}G^{(2)}_{tr})
(\prod_{r=1}^{s-1}\widehat\alpha_{2r,2s}G_{sr}^{(2)})\\
&\overset{\text{Lem.\,\ref{l.e2}(3)}}{=}
\be_{2s-1,2s}(y_{2s}z_{2t}
+
\alpha_{2t-1,2t}\alpha_{2t,2s}\alpha_{2s,2s-1}y_{2t}z_{2s})(\prod_{r=1}^{s-1}\widehat\alpha_{2r,2s}G_{sr}^{(2)})(\prod_{r=1}^{t-1}\widehat\alpha_{2r,2t}G^{(2)}_{tr})\\
&\overset{\eqref{e.normale}}{=}
\be_{2s-1,2s}(y_{2s}z_{2t}
+
\alpha_{2t-1,2s-1}y_{2t}z_{2s})(\prod_{r=1}^{s-1}\widehat\alpha_{2r,2s}G_{sr}^{(2)})(\prod_{r=1}^{t-1}\widehat\alpha_{2r,2t}G^{(2)}_{tr})=0.
\end{align*}
If $i=2s-1$ is odd, $j=2t$ is even with $s<t$, then
\begin{align*}
&\psi(y_{i}'z_{j}'+y_{j}'z_{i}')
=(y_{2s-1}\prod_{r=1}^{s-1}F^{(1)}_{sr})
(z_{2t}\prod_{r=1}^{t-1}\widehat\alpha_{2r,2t}G_{tr}^{(2)})
+
(\be_{2t-1,2t}y_{2t}\prod_{r=1}^{t-1}\widehat\alpha_{2r,2s}F^{(2)}_{tr})
(\be_{2s,2s-1}z_{2s-1}\prod_{r=1}^{s-1}G_{sr}^{(1)})\\
&\overset{\text{Lem.\,\ref{l.e2}(2)}}{=}
y_{2s-1}z_{2t}(\prod_{r=1}^{s-1}G^{(1)}_{sr})
(\prod_{r=1}^{t-1}\widehat\alpha_{2r,2t}G_{tr}^{(2)})
+
\be_{2t-1,2t}\be_{2s,2s-1}\widehat\alpha_{2t,2s-1}y_{2t}z_{2s-1}(\prod_{r=1}^{t-1}\widehat\alpha_{2r,2s}G^{(2)}_{tr})
(\prod_{r=1}^{s-1}G_{sr}^{(1)})\\
&\overset{\text{Lem.\,\ref{l.e2}(3)}}{=}
y_{2s-1}z_{2t}(\prod_{r=1}^{s-1}G^{(1)}_{sr})
(\prod_{r=1}^{t-1}\widehat\alpha_{2r,2t}G_{tr}^{(2)})
+
\alpha_{2t-1,2t}\alpha_{2t,2s-1}y_{2t}z_{2s-1}(\prod_{r=1}^{s-1}G_{sr}^{(1)})(\prod_{r=1}^{t-1}\widehat\alpha_{2r,2s}G^{(2)}_{tr})
\\
&\overset{\eqref{e.normale}}{=}
(y_{2s-1}z_{2t}+\alpha_{2s,2t}y_{2t}z_{2s-1})(\prod_{r=1}^{s-1}G_{sr}^{(1)})(\prod_{r=1}^{t-1}\widehat\alpha_{2r,2t}G_{tr}^{(2)})=0.
\end{align*}
If $i=2s-1$ is odd, $j=2t$ is even with $s>t$, then
\begin{align*}
&\psi(y_{i}'z_{j}'+y_{j}'z_{i}')
=(y_{2s-1}\prod_{r=1}^{s-1}F^{(1)}_{sr})
(z_{2t}\prod_{r=1}^{t-1}\widehat\alpha_{2r,2t}G_{tr}^{(2)})
+
(\be_{2t-1,2t}y_{2t}\prod_{r=1}^{t-1}\widehat\alpha_{2r,2s}F^{(2)}_{tr})
(\be_{2s,2s-1}z_{2s-1}\prod_{r=1}^{s-1}G_{sr}^{(1)})\\
&\overset{\text{Lem.\,\ref{l.e2}(2)}}{=}
\widehat\alpha_{2t,2s}y_{2s-1}z_{2t}(\prod_{r=1}^{s-1}G^{(1)}_{sr})
(\prod_{r=1}^{t-1}\widehat\alpha_{2r,2t}G_{tr}^{(2)})
+
\be_{2t-1,2t}\be_{2s,2s-1}y_{2t}z_{2s-1}(\prod_{r=1}^{t-1}\widehat\alpha_{2r,2s}G^{(2)}_{tr})
(\prod_{r=1}^{s-1}G_{sr}^{(1)})\\
&\overset{\text{Lem.\,\ref{l.e2}(3)}}{=}
\widehat\alpha_{2t,2s}y_{2s-1}z_{2t}(\prod_{r=1}^{s-1}G^{(1)}_{sr})
(\prod_{r=1}^{t-1}\widehat\alpha_{2r,2t}G_{tr}^{(2)})
+
\be_{2t-1,2t}\be_{2s,2s-1}y_{2t}z_{2s-1}(\prod_{r=1}^{s-1}G_{sr}^{(1)})(\prod_{r=1}^{t-1}\widehat\alpha_{2r,2s}G^{(2)}_{tr})
\\
&=
\widehat\alpha_{2t,2s}(y_{2s-1}z_{2t}+\alpha_{2s,2t}y_{2t}z_{2s-1})(\prod_{r=1}^{s-1}G_{sr}^{(1)})(\prod_{r=1}^{t-1}\widehat\alpha_{2r,2t}G_{tr}^{(2)})=0.
\end{align*}
Using the same reasoning, we can prove that 
$\psi(z_{i}'y_{j}'+z_{j}'y_{i}')=0$.

It remains to check that the remaining three relations vanish under $\psi$. For $1\leq s\leq m$, by Lemma \ref{l.e2}(2),(3),
\begin{align*}
\psi(y_{2s-1}') \psi(z_{2s}')&=
(y_{2s-1}\prod_{r=1}^{s-1}F^{(1)}_{sr})(z_{2s}\prod_{r=1}^{s-1}\widehat\al_{2s,2r}G_{sr}^{(2)})
=
y_{2s-1}z_{2s}(\prod_{r=1}^{s-1}\widehat\al_{2s,2r}G^{(1)}_{sr}G_{sr}^{(2)}),\\
\psi(y_{2s}') \psi(z_{2s-1}')&=(\be_{2s-1,2s}y_{2s}\prod_{r=1}^{s-1}\widehat\al_{2s,2r}F_{sr}^{(2)})(\be_{2s,2s-1}z_{2s-1}\prod_{r=1}^{s-1}G_{sr}^{(1)})
=
y_{2s}z_{2s-1}(\prod_{r=1}^{s-1}\widehat\al_{2s,2r}G_{sr}^{(1)}G^{(2)}_{sr}).
\end{align*}
Thus we have
\begin{align*}
&\psi(y_{2s-1}'z_{2s}'+y_{2s}'z_{2s-1}')=
(y_{2s-1}z_{2s}+y_{2s}z_{2s-1})(\prod_{r=1}^{s-1}\widehat\al_{2r,2s}G_{sr}^{(1)}G^{(2)}_{sr})
\overset{\eqref{e.ip}}{=}E(\prod_{r=1}^{s-1}\widehat\al_{2r,2s}G_{sr}^{(1)}G^{(2)}_{sr})\\
&\overset{\text{Lem.\,\ref{l.e2}(4)}}{=}E(\prod_{r=1}^{s-1}\widehat\al_{2r,2s}(1-(1-\widehat\alpha_{2s,2r})E)\overset{(\star)}{=}E(\prod_{r=1}^{s-1}\widehat\al_{2r,2s}(1-(1-\widehat\alpha_{2s,2r}))
=E(\prod_{r=1}^{s-1}\widehat\al_{2r,2s}\widehat\alpha_{2s,2r})=E.
\end{align*}
Here $(\star)$ follows from the fact that $E$ is an idempotent. Similar arguments imply $\psi(z_{2s-1}'y_{2s}'+z_{2s}'y_{2s-1}')=E'$.
Therefore,
\begin{align*}
&\psi(y_{2s-1}'z_{2s}'+y_{2s}'z_{2s-1}'-(y_{1}'z_{2}'+y_{2}'z_{1}'))=E-E=0,\\
&\psi(z_{2s-1}'y_{2s}'+z_{2s}'y_{2s-1}'-(z_{1}'y_{2}'+z_{2}'y_{1}'))=E'-E'=0,\\
&\psi(y_{1}'z_{2}'+y_{2}'z_{1}'+z_{1}'y_{2}'+z_{2}'y_{1}'-1)=E+E'-1=0.
\end{align*}

Hence we get the induced algebra homomorphism $\overline{\psi}: C(A_{\textnormal{comm}}) \to C(A_\al)$.

Next, we show by induction that $y_i, z_i \in \Im \overline{\psi}$ for all $i$, which implies that $\overline{\psi}$ is surjective.
Clearly, $y_1 = \overline{\psi}(y_1'), y_2 = \overline{\psi}(\be_{21}y_2'), z_1 = \overline{\psi}(\be_{12}z_1'), z_2 = \overline{\psi}(z_2')  \in \Im \overline{\psi}$. Suppose that $y_i, z_i \in \Im \overline{\psi}$ for all $1\leq i\leq 2\ell-2$. Then $F^{(a)}_{\ell r}, G^{(a)}_{\ell r} \in \Im \overline{\psi}$ for all $1\leq r \leq \ell-1$ and $a \in \{1,2\}$, so 
\begin{align*}
(\prod_{r=1}^{\ell-1}\widehat\alpha_{2\ell,2r})y_{2\ell-1} 
&=y_{2\ell-1}(\prod_{r=1}^{\ell-1}\widehat\alpha_{2\ell,2r}+(1-\widehat\alpha_{2\ell,2r})E')
=y_{2\ell-1}(\prod_{r=1}^{\ell-1}1-(1-\widehat\alpha_{2\ell,2r})(1-E'))\\
&=y_{2\ell-1}(\prod_{r=1}^{\ell-1}1-(1-\widehat\alpha_{2\ell,2r})E)
\overset{\text{Lem.\,\ref{l.e2}(4)}}{=}
y_{2\ell-1}(\prod_{r=1}^{\ell-1}F^{(1)}_{\ell r}F^{(2)}_{\ell r})\\
&\overset{\text{Lem.\,\ref{l.e2}(3)}}{=} \overline{\psi}(z_{2\ell-1}')(\prod_{r=1}^{\ell-1}F^{(2)}_{\ell r}) \in \Im \overline{\psi},\\
(\prod_{r=1}^{\ell-1}\widehat\al_{2\ell,2r})y_{2\ell}
&=y_{2\ell}(\prod_{r=1}^{\ell-1}\widehat\alpha_{2\ell,2r}+(1-\widehat\alpha_{2\ell,2r})E')
=y_{2\ell}(\prod_{r=1}^{\ell-1}1-(1-\widehat\alpha_{2\ell,2r})(1-E'))\\
&=y_{2\ell}(\prod_{r=1}^{\ell-1}1-(1-\widehat\al_{2\ell,2r})E)
\overset{\text{Lem.\,\ref{l.e2}(4)}}{=}
y_{2\ell}(\prod_{r=1}^{\ell-1}F^{(2)}_{\ell r}F^{(1)}_{\ell r})\\
&\overset{\text{Lem.\,\ref{l.e2}(3)}}{=}\be_{2\ell,2\ell-1}(\prod_{r=1}^{\ell-1}\widehat\al_{2\ell,2r})\overline{\psi}(y_{2\ell}')(\prod_{r=1}^{\ell-1}F^{(1)}_{\ell r}) \in \Im \overline{\psi}.
\end{align*}
Therefore, $y_{2\ell-1},y_{2\ell} \in \Im \overline{\psi}$. An analogous argument shows that $z_{2\ell-1},z_{2\ell} \in \Im \overline{\psi}$. Hence $\overline{\psi}$ is surjective.
Since $\dim_k C(A_{\textnormal{comm}})= \dim_k C(A_\al) =2^{2m+1}$, it follows that 
$\overline{\psi}$ is an isomorphism.
\end{proof}

We now prove Theorem \ref{t.intromain} in the case where $n$ is even.

\begin{proof}[Proof of Theorem \ref{t.intromain} for even $n$]

Since $A_{\textnormal{comm}} \cong k[x_1,\dots,x_n]/(x_1^2+\dots+x_n^2)=:B$,
it follows from Theorem \ref{t.eiso} and \eqref{e.commCA} that $C(A_\al) \cong C(A_{\textnormal{comm}}) \cong C(B) \cong M_{2^{(n-2)/2}}(k)^2$.
Furthermore, Theorem \ref{t.SV} and Morita theory imply that
$\uCM(A_\al) \simeq \D^b(\mod  (M_{2^{(n-2)/2}}(k)^2)^{\op})\simeq \D^b(\mod k^2)$.
\end{proof}

\section{Proof of Corollary \ref{c.intromain} and an Example}\label{sec.ce}

We give a proof of Corollary~\ref{c.intromain} here.

\begin{proof}[Proof of Corollary \ref{c.intromain}]
(1) By Theorem \ref{t.intromain}, $C(S_\alpha/(f))$ is a semisimple algebra, and hence it follows from \cite[Theorem 5.5]{MUk} that $S_\alpha/(f)$ is of finite Cohen-Macaulay representation type.
Moreover, by the proof of \cite[Theorem 5.5]{MUk}, the number of non-projective graded maximal Cohen-Macaulay modules over $S_\alpha/(f)$, up to isomorphism and degree shift, is equal to the number of isomorphism classes of simple $C(S_\alpha/(f))$-modules.
Hence, this number is one if $n$ is odd and two if $n$ is even.

(2) Since $C(S_\alpha/(f))$ is semisimple, \cite[Theorem 5.5]{MUk} also implies that $\qgr S_\alpha/(f)$ has finite global dimension.
It follows from \cite[Theorem A.4]{dV} that $\D^b(\qgr S_\alpha/(f))$ admits a Serre functor and that
$\gldim(\qgr S_\alpha/(f))=n-2$.
\end{proof}

In the commutative case, it is well-known that there exists a close relationship between maximal Cohen-Macaulay modules over a hypersurface $S/(f)$ and matrix factorizations of $f \in S$.
By using twisted matrix factorizations \cite{CKMW} or noncommutative matrix factorizations \cite{MUm}, it is known that an analogue of this correspondence also holds when $S$ is an AS-regular algebra and $f$ is a homogeneous regular normal element.
Here, as an example, we describe a non-projective indecomposable maximal Cohen-Macaulay module over $A_\al$ with $n=5$ using noncommutative matrix factorizations; note that, by Corollary~\ref{c.intromain}, it is unique up to isomorphism and degree shift.

\begin{ex}
Let us consider a skew polynomial algebra
$S := k\langle x_1,\dots,x_5\rangle /(x_i x_j - \alpha_{ij} x_j x_i)$
in five variables, and let
$f = x_1 x_2 + x_3 x_4 + x_5^2 \in S$.
Assume that $f$ is normal. This is equivalent to the following conditions:
\[
\alpha_{12}=\alpha_{13}\alpha_{14}=\alpha_{32}\alpha_{42}=\alpha_{15}^2,\quad
\alpha_{34}=\alpha_{31}\alpha_{32}=\alpha_{14}\alpha_{24}=\alpha_{35}^2,\quad \textnormal{and} \quad
\alpha_{51}\alpha_{52}=\alpha_{53}\alpha_{54}=1.
\]
We set $\alpha := \alpha_{15}$, $\beta := \alpha_{35}$, and $\gamma := \alpha_{13}$.
Then all $\alpha_{ij}$ are determined by $\alpha, \beta, \gamma$.
More precisely, the defining relations of $S$ are given by
\begin{align*}
&x_{1}x_{2}-\alpha^2 x_{2}x_{1}, 
&&x_{1}x_{3}-\gamma x_{3}x_{1},
&&x_{1}x_{4}-\alpha^2\gamma^{-1} x_{4}x_{1}, 
&&x_{1}x_{5}-\alpha x_{5}x_{1}, 
&&x_{2}x_{3}-\beta^{-1}\gamma^{-1} x_{3}x_{2}, \\
&x_{2}x_{4}-\alpha^{-2}\beta^2\gamma x_{4}x_{2}, 
&&x_{2}x_{5}-\alpha^{-1} x_{5}x_{2},
&&x_{3}x_{4}-\beta^2 x_{4}x_{3},
&&x_{3}x_{5}-\beta x_{5}x_{3},
&&x_{4}x_{5}-\beta^{-1} x_{5}x_{4}.
\end{align*}

By Theorem~\ref{t.intromain}, we have $C(S/(f)) \cong M_4(k)$.
For each $i \in \mathbb{Z}$, define
\[
\Phi^i =
\begin{bmatrix}
x_5 & \beta^{-i}x_4 & \alpha^{-i}x_2 & 0\\
\beta^{i+1}x_3 & -x_5 & 0 & \alpha^{-i-1}\gamma x_2\\
\alpha^{i+1}x_1 & 0 & -x_5 & -\beta^{-i-1}x_4\\
0 & \alpha^{i+2}\gamma^{-1}x_1 & -\beta^{i+2}x_3 & x_5
\end{bmatrix}
\in M_4(S).
\]
A direct computation shows that
\[
\Phi^i \Phi^{i+1} = f I_4,
\]
where $I_4$ denotes the identity matrix in $M_4(S)$.
Thus $\{\Phi^i\}_{i \in \ZZ}$ defines a noncommutative right matrix factorization of $f$ over $S$; see \cite[Definition~2.1 and Remark~2.2(2)]{MUm}.

Since the simple $M_4(k)$-module has dimension $4$ over $k$, it follows from the proof of \cite[Lemma~5.11]{MUk} that a non-projective indecomposable maximal Cohen-Macaulay module over $S/(f)$ is obtained from a noncommutative right matrix factorization of $f$ of rank $4$.
Hence
\[
X := \Coker \phi, \qquad
\textnormal{where }
\phi \colon S(-1)^4 \to S^4;\ \phi(a)=\Phi^0 a,
\]
is a non-projective indecomposable maximal Cohen-Macaulay module over $S/(f)$; see \cite[Proposition~5.10]{MUm}.
Therefore, every non-projective indecomposable maximal Cohen-Macaulay module over $S/(f)$ is isomorphic to $X$ up to degree shift.
For example, if we define
$\Psi^i := \Phi^{i+1}$ for $i \in \ZZ$, then $\{\Psi^i\}_{i \in \ZZ}$ defines another noncommutative right matrix factorization of $f$ over $S$. However,
\[
P \Phi^i P^{-1} = \Psi^i \textnormal{ for all } i \in \ZZ,
\qquad \text{where }
P =
\begin{bmatrix}
1 & 0 & 0 & 0\\
0 & \beta & 0 & 0\\
0 & 0 & \alpha & 0\\
0 & 0 & 0 & \alpha\beta
\end{bmatrix},
\]
so $\{\Phi^i\}_{i \in \ZZ}$ and $\{\Psi^i\}_{i \in \ZZ}$ are isomorphic as right matrix factorizations.
Hence the corresponding maximal Cohen-Macaulay modules are isomorphic.
\end{ex}

\section*{Acknowledgment}
The second author was supported by JSPS KAKENHI Grant Number JP22K03222.

\end{document}